\documentclass{article}
\usepackage{graphicx} 
\usepackage{amsfonts, amsmath, amssymb, amsthm}
\usepackage{hyperref}
\usepackage{url} 
\usepackage{color, soul ,esint,enumitem}

\newtheorem{definition}{Definition}
\newtheorem{remark}{Remark}
\newtheorem{example}{Example}
\newtheorem{theorem}{Theorem}
\newtheorem{proposition}{Proposition}
\newtheorem{lemma}{Lemma}

\newcommand{\N}{\mathbb{N}}
\newcommand{\Z}{\mathbb{Z}}
\newcommand{\R}{\mathbb{R}}
\newcommand{\inc}[1]{\hyperref[def:Inc]{$(\mathrm{Inc})_{#1}$}}
\newcommand{\dec}[1]{\hyperref[def:Dec]{$(\mathrm{Dec})_{#1}$}}

\title{On generalized Sobolev-Orlicz spaces associated to the Riesz fractional gradient}
\author{Pedro Miguel Campos\footnote{CMAFcIO – Departamento de Matemática, Faculdade de Ciências, Universidade de Lisboa P-1749-016 Lisboa, Portugal \\Email address: \texttt{pmcampos@fc.ul.pt}}}
\date{}

\begin{document}

\maketitle

\begin{abstract}
    We introduce a new family of function spaces, the fractional generalized Sobolev-Orlicz spaces $\Lambda^{s,A}_0(\Omega)$, where $A$ is a generalized $\Phi$-function satisfying the $(\mathrm{Inc})_{p}$ and $(\mathrm{Dec})_{q}$ conditions for $1<p\leq q<\infty$, as an extension of the Lions-Calderón spaces (also known as Bessel potential spaces) $\Lambda^{s,p}_0(\Omega)$ when $0<s<1$ to the generalized Orlicz framework. We obtain some continuous and compact embeddings for these spaces and study the continuous dependence of the Riesz fractional gradient $D^s$ with respect to  $s\in[0,1]$ as $s\to \sigma\in[0,1]$. Finally, we apply these results to study the existence, uniqueness and continuous dependence of a family of partial differential equations depending on the Riesz fractional gradient as $s\to\sigma\in(0,1]$.
\end{abstract}

\section{Introduction}
In the last years there has been some interest in the study of fractional differential operators, and their applications. One such operator is the Riesz fractional gradient $D^s$, originally defined in \cite{shieh2015on, silhavy2020Fractional} and further studied in \cite{bellido2020gamma, comi2019BlowUp, comi2022asymptoticsI, brue2022AsymptoticsII, kreisbeck2022Quasiconvexity}. This operator is, up to multiplication by a constant, the only that transforms scalar-valued into vector-valued functions such that it is rotationally and translationally invariant, $s$-homogeneous and continuous in the sense of distributions (see \cite{silhavy2020Fractional}). Moreover, they generate the fractional Laplacian like the classical gradient generates the classical Laplacian, i.e.,
\begin{equation}
-D^s\cdot D^s u=(-\Delta)^s u.
\end{equation}
These properties suggest that $D^s$ is a good generalization of the classical gradient $D$ when $s<1$, therefore, we make the identification of $D^1=D$. On the other hand, when $s=0$, one makes the identification $D^0=-R$ where $R$ is the vector-valued Riesz transform, see \cite{brue2022AsymptoticsII}.
Due to these properties, some authors have suggested, for example in \cite{bellido2020Piola}, that such operators might be useful to study certain phenomena in physics, namely in nonlocal theories of elasticity like peridynamics. 
This led to some interesting work regarding the existence of solutions to fractional partial differential equations, associated to potentials in \cite{bellido2020gamma} and more general operators in \cite{campos2023unilateral}. 
However, it is important to remark that all the work concerning these operators that has been done, at least to the author’s knowledge, revolves around the classic Lebesgue spaces and their corresponding fractional version of the Sobolev spaces, $\Lambda^{s,p}_0(\Omega)$ (some authors also use the notation $H^{s,p}_0(\Omega)$). 
But just like in the classical case where Sobolev-Orlicz spaces $W^{1,A}_0(\Omega)$ are used to study nonhomogeneous elastic materials, it also makes sense to generalize these spaces to the fractional setting so that it can be eventually used to study such materials in the nonlocal framework mentioned above. 

The goal of this article is to provide a new functional analytic framework to solve nonlinear partial differential equations that make use of the Riesz fractional gradient.
To do so, in Section \ref{sec:generalized_orlicz_spaces} we recall the notion of $\Phi$-functions and introduce some of the technical hypothesis, that we further impose on this class of functions, based on \cite{harjulehto2019orlicz, harjulehto2023revised}.
In particular, we assume that our $\Phi$-functions are \inc{p} and \dec{q} for some $1<p\leq q<\infty$, and that they satisfy three technical hypothesis \hyperref[def:a0]{(A0)}, \hyperref[def:a1]{(A1)} and \hyperref[def:a2]{(A2)}. 
The \inc{p} and \dec{q} describe the growth of the $\Phi$-function and its conjugate, and are equivalent to the $\Delta_2$ and $\nabla_2$ conditions, while the other three assumptions are important to apply arguments of harmonic analysis to analyze these functions.
After that, we recall the notion of generalized Orlicz spaces and some of its properties. 
In Section \ref{sec:fractional_generalized_sobolev_orlicz_spaces} we introduce the fractional generalized Sobolev-Orlicz spaces $\Lambda^{s,A}_0(\Omega)$ which correspond to $W^{1,A}_0(\Omega)$ when $s=1$, and to $\Lambda^{s,p}_0(\Omega)$ when the $\Phi$-function is $A(r)=\frac{1}{p}r^{p}$.
Then, we provide results concerning the continuous and compact embeddings, which include fractional versions of the Sobolev and Poincaré inequalities, as well as the continuous dependence of $D^s$ with respect to the fractional parameter $s$. 
Finally, in Section \ref{sec:dirichlet_problem} we apply the results proved in the previous sections to study the existence of solutions in $\Lambda^{s,A}_0(\Omega)$ to the problem
\begin{equation}\label{eq:dirichlet_problem}
    \begin{cases}
        -D^s\cdot(a(x,D^s u)D^s u)=F & \mbox{ in }\Omega\\
        u=0 & \mbox{ on } \R^d\setminus \Omega,
    \end{cases}
\end{equation}
where $a$ is related to $A$ by the following identity
\begin{equation}\label{eq:relation_between_a_and_A}
    A(x,\ell)=\int_0^\ell{a(x,r)r}\,dr,
\end{equation}
as well as the continuous dependence of these solutions with respect to $s$, as $s\to\sigma\in(0,1]$. These generalize the existence and continuous dependence results obtained in \cite{campos2023unilateral} for equations in the $\Lambda^{s,p}_0(\Omega)$-framework.

\section{Generalized Orlicz spaces}\label{sec:generalized_orlicz_spaces}
We start this section by recalling the definition of the class of $\Phi$-functions. This class of functions, appears in the literature under the name of N-functions or Young functions, but we have decided to use the term $\Phi$-function because it is the terminology of \cite{harjulehto2019orlicz, harjulehto2023revised} which are the two main references that we follow for further assumptions on these functions throughout this text. 
Since for our analysis, we use the hypothesis that are expressed in \cite{harjulehto2019orlicz}, we also use their nomenclature of generalized strong $\Phi$-function (or simply $\Phi$-function), to address these functions. 

\begin{definition}[$\Phi$-function] Let $\Omega\subset\R^d$ be an open set. A function $A:\Omega\times[0,\infty)\to[0,\infty)$ is called a (generalized) $\Phi$-function, and is denoted $A\in\Phi(\Omega)$, if:
    \begin{itemize}
        \item[i)] $x\mapsto A(x,|f(x)|)$ is measurable in $\Omega$ for every measurable (scalar or vector-valued) function $f$ defined on $\Omega$;
        \item[ii)] $A(x,\cdot)$ is increasing, convex and satisfies $A(x,0)=0$, $\lim_{\ell\to0^+}{A(x,\ell)}=0$ $\lim_{\ell\to+\infty}{A(x,\ell)}=\infty$ for a.e. $x\in\Omega$.
    \end{itemize}
\end{definition}

\begin{remark}
    There are weaker notions of $\Phi$-functions, namely weak and convex $\Phi$-functions.
    These generalizations, which we do not make use of here, are useful when one wants to consider discontinuous functions $A$. 
    For more information, see \cite[Definition 2.5.2]{harjulehto2019orlicz} and the discussion that follows.
\end{remark}

Throughout this manuscript, we require further hypothesis on the $\Phi$-functions.
We collect them in the following definition.

\begin{definition}[Hypothesis on $A$]
    Let $\Omega\subset \R^d$ be an open set, $A\in\Phi(\Omega)$ and $p>1$ and $q<\infty$. Let $A^{-1}(x,\ell)$ denotes the left inverse of $A$, i.e.,
    \begin{equation*}
        A^{-1}(x,r):=\inf\{\ell\geq 0:\, A(x,\ell)\geq r\}.
    \end{equation*}
    We say that $A$ satisfies:
    \begin{itemize}
        \item[$(\mathrm{Inc})_p$]\label{def:Inc} if function $\ell\mapsto \frac{A(x,\ell)}{\ell^p}$ is increasing for a.e. $x\in\Omega$;
        \item[$(\mathrm{Dec})_q$]\label{def:Dec} if the function $\ell\mapsto \frac{A(x,\ell)}{\ell^q}$ is decreasing for a.e. $x\in\Omega$;
        \item[(A0)]\label{def:a0} if there exists a constant $\beta\in (0,1]$ such that $A(x,\beta)\leq 1\leq A(x,1/\beta)$ for a.e. $x\in\Omega$;
        \item[(A1)]\label{def:a1} if there exists a constant $0<\beta\leq 1$ such that
        \begin{equation*}
            \beta A^{-1}(x,\ell)\leq A^{-1}(y,\ell)
        \end{equation*}
        for every $1\leq \ell\leq 1/|B|$, a.e. $x,y\in B\cap\Omega$ and every ball $B$ with $|B|\leq 1$;
        \item[(A2)]\label{def:a2} if for every $\sigma>0$ there exists $\beta\in(0,1]$ and $h\in L^1(\Omega)\cap L^\infty(\Omega)$ with $h\geq 0$, such that for a.e. $x,y\in\Omega$,
        \begin{equation*}
            \beta A^{-1}(x,\ell)\leq A^{-1}(y,\ell+h(x)+h(y))
        \end{equation*}
        whenever $0\leq\ell\leq\sigma$.
    \end{itemize}
\end{definition}

The hypothesis \hyperref[def:Inc]{$(\mathrm{Inc})_p$} and \hyperref[def:Dec]{$(\mathrm{Dec})_q$} are linked to the the growth of the $\Phi$-functions which are important for the study of partial differential equations that depend on them. 
In fact, they are related with the well-known $\Delta_2$-condition. 
More precisely, any $\Phi$-function $A$ that satisfies the growth condition \hyperref[def:Dec]{$(\mathrm{Dec})_q$} for some $q<\infty$, also satisfies the $\Delta_2$-condition, i.e., there exists $K\geq 2$ such that 
\begin{equation*}
    A(x,2\ell)\leq KA(x,\ell) \qquad\mbox{ for a.e. } x\in\Omega \mbox{ and all } \ell\geq 0.
\end{equation*}
Similarly, if $A$ satisfies \inc{p} for some $p>1$, its conjugate, i.e. the function
\begin{equation}\label{eq:convex_conjugate}
    A'(x,\ell):=\sup\{r\ell-A(x,r):\, r\geq 0\},
\end{equation}
satisfies the $\Delta_2$-condition. 
Moreover, it is prove in \cite[Lemma 2.2.6 (b)]{harjulehto2019orlicz} and \cite[Corollary 2.4.11]{harjulehto2019orlicz}, that both implications are in fact equivalences, that is, if $A$ satisfies the $\Delta_2$-condition then it is \dec{q} for some $q<\infty$, and if $A'$ satisfies the $\Delta_2$-condition then $A$ is \inc{p} for some $p>1$.

\begin{remark}
    For applications to partial differential equations, one usually considers $\Phi$-functions $A$ that are related to a function $a(x,r)$  through \eqref{eq:relation_between_a_and_A}. In this case, if $a$ is measurable and bounded function in $x$ for all $r>0$ and Lipschitz continuous in $r$ for a.e. $x\in\Omega$, and there exist constants $a_-,a_+>0$ such that
    \begin{equation}\label{eq:hypothesis_on_a}
        a_-\leq \frac{r \partial_r a(x,r)}{a(x,r)}+1\leq a_+,
    \end{equation}
    then $A$ is \inc{a_- +1} and \dec{a_+ +1}. This well-known result is possible to prove by multiplying each term of \eqref{eq:hypothesis_on_a} by $r$, integrating with respect to $r$ ranging from $0$ to $\ell$, and then comparing with the derivatives with respect to $\ell$ of $\frac{A(x,\ell)}{\ell^{a_- +1}}$ and of $\frac{A(x,\ell)}{\ell^{a_+ +1}}$. Moreover, through the process  of this proof one also obtains that
    \begin{equation}\label{eq:pointwise_relationship_between_a_and_A}
        (a_- + 1)A(x,\ell)\leq \ell^2 a(x,\ell)\leq (a_+ + 1)A(x,\ell).
    \end{equation}
\end{remark}

The condition \hyperref[def:a0]{(A0)} restricts our $\Phi$-functions to the ``unweighted'' case. 
To better understand what this means, we note that when one is dealing with the $\Phi$-function $A(x,\ell)=\alpha(x)\ell^p$, this hypothesis imposes that the weight $\alpha(x)$ is comparable to $1$, i.e., there exists a constant $C>1$ such that $1/C\leq \alpha(x)\leq C$ for a.e. $x\in\R^d$.

The last two hypothesis \hyperref[def:a1]{(A1)} and \hyperref[def:a2]{(A2)} are technical conditions needed to apply results from harmonic analysis to the study of $\Phi$-functions.
In particular, they help us establish continuous and compact embeddings for the fractional generalized Sobolev-Orlicz spaces that we introduce in the next section.
It is however important to remark that the definition that we give for \hyperref[def:a2]{(A2)} is the one given in \cite{harjulehto2023revised} and not the one from \cite{harjulehto2019orlicz}. 
The reason why we do this is because according to \cite{harjulehto2023revised}, the definition of (A2) in \cite{harjulehto2019orlicz} contains a flaw.
Nevertheless, the results from \cite{harjulehto2019orlicz} that we make use of, hold true for the definition of \hyperref[def:a2]{(A2)} given here.

This list of hypothesis, allow us to work with the following families of $\Phi$-functions:

\begin{example}[Variable order exponent]
    Consider a function $p(x):\R^d\to [1,\infty]$ such that $p^-=\mathrm{ess\,inf}_{x\in\R^d}{p(x)}>1$ and $p^+=\mathrm{ess\, sup}_{x\in\R^d}{p(x)}<\infty$. Assume also that $1/p$ is log-Hölder continuous and satisfies the Nekvinda's decay condition, i.e.,
    \begin{equation*}
        \left|\frac{1}{p(x)}-\frac{1}{p(y)}\right|\leq \frac{C}{\log{(e+\frac{1}{|x-y|})}},\quad\forall x\neq y \mbox{ in } \R^d
    \end{equation*}
    and
    \begin{equation*}
        \exists c>0 \mbox{ and } p_\infty\in[1,\infty]: \quad\int_{\{p(x)\neq p_\infty\}}{c^\frac{1}{|\frac{1}{p(x)}-\frac{1}{p_\infty}|}}\,dx<\infty
    \end{equation*}
    respectively.
    Assume that $a(x,r)=\alpha(x)p(x)r^{p(x)-2}$ with $\frac{1}{c}\leq \alpha(x)\leq c$ for some $c\geq 1$ for a.e. $x\in\R^d$. Then $A(x,\ell)=\alpha(x)\ell^{p(x)}\in\Phi(\R^d)$ satisfies \inc{p_-}, \dec{p_+}, \hyperref[def:a0]{(A0)}, \hyperref[def:a1]{(A1)} and \hyperref[def:a2]{(A2)} by \cite[Lemma 7.1.1, Proposition 7.1.2 and Proposition 7.1.3]{harjulehto2019orlicz}.
\end{example}

\begin{example}[Logarithmic perturbation of the variable order exponent]
    Assuming the same hypothesis on $p(\cdot)$ as in previous example, and let $a(x,r)=p(x)r^{p(x)-2}\log{(e+r)}+t^{p(x)-1}/(e+t)$. Then, $A(x,\ell)=\ell^{p(x)}\log{(e+r)}\in\Phi(\R^d)$ satisfies \inc{p_-}, \dec{p_+}, \hyperref[def:a0]{(A0)}, \hyperref[def:a1]{(A1)} and \hyperref[def:a2]{(A2)} as one can see in \cite[Table 7.1]{harjulehto2019orlicz}.  
\end{example}

\begin{example}[Double phase]
    Let $a(x,r)=pr^{p-2}+q\alpha(x)r^{q-2}$ for some $1<p<q<\infty$ and with $\alpha\in L^\infty(\R^d)\cap C^{\frac{d}{p}(q-p)}(\R^d)$. Then $A(x,\ell)=\ell^{p}+\alpha(x)\ell^{q}$ satisfies \inc{p}, \dec{q}, \hyperref[def:a0]{(A0)}, \hyperref[def:a1]{(A1)} and \hyperref[def:a2]{(A2)} by \cite[Propositions 7.2.1 and 7.2.2]{harjulehto2019orlicz}.
\end{example}

\vspace{2mm}

With this short presentation of the $\Phi$-functions that important for this work, we recall the definition of the generalized Orlicz spaces.

\begin{definition}[Generalized Orlicz space]
    Let $\Omega\subset \R^d$ be an open set and $A\in \Phi(\Omega)$. We define the generalized Orlicz space
    \begin{equation*}
        L^A(\Omega)=\{f:\Omega\to\R \mbox{ measurable}:\, \int_{\Omega}{A(x,|\rho f(x)|)}\,dx<\infty \mbox{ for some } \rho >0\}.
    \end{equation*}
    We endow this space with the Luxemburg norm
    \begin{equation*}
        \|f\|_{L^A(\Omega)}=\inf\left\{\rho>0:\int_{\Omega}{A\left(x,\left|\frac{f(x)}{\rho}\right|\right)}\,dx\leq 1\right\}
    \end{equation*}
    When dealing with vector-valued functions, we use the notation $L^A(\Omega;\R^d)$ for the respective generalized Orlicz space.
\end{definition}

Since we assume that $A$ and $A'$ satisfy the $\Delta_2$ condition, there is a relationship between the modular $J_A$ and the norm $\|\cdot\|_{L^A(\R^d)}$.

\begin{lemma}\label{lemma:relation_norm_modular}
    Let $A\in\Phi(\R^d)$, satisfying \hyperref[def:Inc]{$(\mathrm{Inc})_p$} and \hyperref[def:Dec]{$(\mathrm{Dec})_q$} with $1<p< q<\infty$. Then, for every $\xi\in L^A(\R^d)$,
    \begin{equation}\label{eq:relation_norm_modular}
        \min\left\{J_A(\xi)^{1/p}, J_A(\xi)^{1/q}\right\}\leq \|\xi\|_{L^A(\R^d)}\leq \max\left\{J_A(\xi)^{1/p}, J_A(\xi)^{1/q}\right\}\leq J_A(\xi)+1
    \end{equation}
    and 
    \begin{equation}\label{eq:reverse_relation_norm_modular}
         \frac{1}{2}\min\{\|\xi\|^{p}_{L^A(\R^d)},\|\xi\|^q_{L^A(\R^d)}\}\leq J_A(\xi)\leq 2 \max\left\{\|\xi\|^p_{L^A(\R^d)}, \|\xi\|^q_{L^A(\R^d)}\right\}
    \end{equation}
    where $J_A(\xi)=\int_{\R^d}{A(x,|\xi|)}\,dx$. The result is also valid for all $\xi\in L^A(\R^d;\R^d)$ with the suitable adaptations.
\end{lemma}
\begin{proof}
    For the two leftmost inequalities of \eqref{eq:relation_norm_modular} see \cite[Lemma 3.2.9]{harjulehto2019orlicz}. The rightmost inequality of \eqref{eq:relation_norm_modular} follows from the fact that when $J_A(\xi)\leq 1$ we have 
    \begin{equation*}
        \max\left\{J_A(\xi)^{1/p}, J_A(\xi)^{1/q}\right\}\leq J_A(\xi)^{1/q}\leq 1
    \end{equation*}
    and when $J_A(\xi)\geq 1$ we have 
    \begin{equation*}
        \max\left\{J_A(\xi)^{1/p}, J_A(\xi)^{1/q}\right\}\leq J_A(\xi)^{1/p}\leq J_A(\xi).
    \end{equation*}
    The same argument of analysing the cases in which $J_A(\xi)$ is smaller or bigger than one, allows us to obtain \eqref{eq:reverse_relation_norm_modular} from \eqref{eq:relation_norm_modular}.
\end{proof}
\begin{remark}
    When the $\Phi$-function $A$ satisfies \inc{p} and \dec{q} with $p=q$, then it follows that $A(x,r)=Cr^p$ for some constant $C>0$. In this case, the we have $\|f\|^p_{L^p(\R^d)}=\left(\int_\Omega{|f|^p}\,dx\right)^{1/p}$ which is a stronger relationship between the Luxemburg and the modular than the one obtained in Lemma \ref{lemma:relation_norm_modular} for the general case.
\end{remark}

Another consequence of the growth conditions \inc{p} and \dec{q} imposed on $A$ is that it allows us to compare the generalized Orlicz spaces with the usual Lebesgue spaces.

\begin{lemma}\label{lemma:relation_classical_and_orlicz}
    Let $A\in \Phi(\R^d)$ satisfy \hyperref[def:a0]{(A0)}, \hyperref[def:Inc]{$(\mathrm{Inc})_p$} and \hyperref[def:Dec]{$(\mathrm{Dec})_q$} for some $1<p< q< \infty$. Then
    \begin{equation*}
        L^p(\R^d)\cap L^q(\R^d)\subset L^A(\R^d)\subset L^p(\R^d)+L^q(\R^d).
    \end{equation*}
    Moreover, if $A\in \Phi(\R^d)$ only satisfies \hyperref[def:a0]{(A0)} and \hyperref[def:Inc]{$(\mathrm{Inc})_p$}, and we restrict it to a bounded open set $\Omega\subset\R^d$, then
    \begin{equation*}
        L^q(\Omega)\subset L^A(\Omega)\subset L^p(\Omega).
    \end{equation*}
    Both chains of inclusions are also true in the vector-valued case.
\end{lemma}
\begin{proof}
    See \cite[Lemma 3.7.7]{harjulehto2019orlicz} for the first assertion and \cite[Corollary 2.7.9]{harjulehto2019orlicz} for the second.
\end{proof}

\section{Fractional generalized Sobolev-Orlicz spaces}\label{sec:fractional_generalized_sobolev_orlicz_spaces}
In this section, we consider a new fractional version of the generalized Sobolev-Orlicz spaces based on \cite{shieh2015on}. 
The idea behind these spaces is that we replace the usage of the (classical) gradient by the fractional Riesz gradient $D^s$, which can be defined for test functions $\varphi\in C^\infty_c(\R^d)$, by
\begin{equation*}
    D^s \varphi=D(I_{1-s}\varphi)
\end{equation*}
where $I_{1-s}$ is the Riesz potential of order $1-s$, $0<1-s<1$,
\begin{equation*}\label{eq:riez_potential}
    I_{1-s}\varphi(x)=(I_{1-s}*\varphi)(x)=\frac{\mu(s,d)}{d+s-1}\int_{\R^d}{\frac{\varphi(y)}{|x-y|^{d+s-1}}}\,dy,
\end{equation*}
with the normalizing constant being given by
\begin{equation*}
    \mu(d,s)=\frac{2^s\Gamma\left(\frac{d+s+1}{2}\right)}{\pi^{d/2}\Gamma\left(\frac{1-s}{2}\right)}.
\end{equation*}
The definition of the normalizing constant $\mu(d,s)$ can be extended naturally (using the same characterization) to the case in which $-1\leq s\leq 1$ with the case $s=1$ being defined through the limit: $\mu(d,1)=\lim_{s\to 1}{\mu(d,s)}=0$. This extension is useful in the formula for the fractional Fundamental Theorem of Calculus described in \cite[Theorem 1.12]{shieh2015on} and \cite[Proposition 15.8]{ponce2016elliptic}, which says that it is possible to recover the function $\varphi$ from the fractional Riesz gradient $D^s \varphi$ through the following formula
\begin{equation}\label{eq:fundamental_theorem_calculus}
    \varphi(x)=\mu(d,-s)\int_{\R^d}{D^s \varphi(y)\frac{x-y}{|x-y|^{d-s+1}}}\,dy=I_s(\mathcal{R}\cdot D^s \varphi)(x),
\end{equation}
where $\mathcal{R}$ is the vector-valued Riesz transform, which is a Calderón-Zygmund singular integral operator in the sense of \cite[p. 115]{harjulehto2019orlicz}. For future reference, we say that an operator $T$ is a Calderón-Zygmund singular integral operator if there exists a kernel $K:(\R^d\times\R^d)\setminus\{(x,x):\, x\in\R\}\to \R$ such that 
\begin{equation*}
    T\varphi(x)=\int_{\R^d}{K_{s}(x,y)\varphi(y)}\,dy, \quad \forall x\notin \mathrm{supp}{(\varphi)} \mbox{ and } \varphi\in C^\infty_c(\R^d)
\end{equation*}
and the kernel $K$ satisfies
\begin{equation}\label{eq:estimate_function_strong_mihklin}
    |K(x,y)|\leq C|x-y|^{-d},
\end{equation}
and for some $\varepsilon>0$, 
\begin{equation}\label{eq:estimate_derivative_strong_mihklin}
    |K(x,y)-K(x,y+h)|+|K(x,y)-K(x+h,y)|\leq C\frac{|h|^\varepsilon}{|x-y|^{d+\varepsilon}}
\end{equation}
for all $h$ such that $|h|\leq \frac{|x-y|}{2}$.

\begin{definition}[Fractional Sobolev-Orlicz space]
    Let us consider $\Omega$ to be an open set, $A\in \Phi(\R^d)$, and $s\in[0,1]$. We define the space $\Lambda^{s,A}_0(\Omega)$ as the completion of $C^\infty_c(\Omega)$, interpreted as functions extended by $0$ to the whole $\R^d$, with respect to the norm
    \begin{equation*}
        \|u\|_{\Lambda^{s,A}_0(\Omega)}=\|u\|_{L^A(\R^d)}+\|D^s u\|_{L^A(\R^d;\R^d)}.
    \end{equation*}
    We also denote by $(\Lambda^{s,A}_0(\Omega))'=\Lambda^{-s,A'}(\Omega)$ the dual of $\Lambda^{s,A}_0(\Omega)$.
\end{definition}

To make sense of the action of the Riesz fractional gradient in these spaces, i.e., $D^s u$ for $u\in \Lambda^{s,A}_0(\Omega)$, we extend its definition by continuity. Indeed, since for each $u\in \Lambda^{s,A}_0(\Omega)$ there exists a sequence $\{u_n\}\subset C^\infty_c(\Omega)$ such that $u_n\to u$ in $L^A(\R^d)$ and $\{D^s u_n\}$ is a Cauchy sequence in $L^A(\R^d;\R^d)$, we define $D^s u$ as the $L^A(\R^d;\R^d)$ limit of the sequence $\{D^s u_n\}$.

As a consequence of this characterization of the functions $D^s u$ with $u\in \Lambda^{s,A}_0(\Omega)$, we have that the duality between the fractional gradient and the fractional divergence still holds in $\Lambda^{s,A}_0(\Omega)$. In fact, for any element $u\in\Lambda^{s,A}_0(\Omega)$ and any sequence of test functions $\{u_n\}\subset C^\infty_c(\Omega)$ such that $u_n\to u$ in $\Lambda^{s,A}_0(\Omega)$, we have that for all $\Psi\in C^\infty_c(\Omega;\R^d)$ 
\begin{equation}
    \int_{\R^d}{D^s u\cdot \Psi}=\lim_{n\to\infty}{\int_{\R^d}{D^s u_n\cdot \Psi}}=-\lim_{n\to\infty}{\int_{\R^d}{u_n D^s\cdot \Psi}}=-\int_{\R^d}{u D^s\cdot \Psi}.
\end{equation}
This duality property, motivates the following characterization of $\Lambda^{-s,A'}(\Omega)$.

\begin{proposition}\label{prop:characterization_dual}
    Let $s\in (0,1]$ and $A\in\Phi(\R^d)$. If $F\in \Lambda^{-s,A'}(\Omega)$, then, there exist functions $f\in L^{A'}(\Omega)$ and $\boldsymbol{f}\in L^{A'}(\R^d;\R^d)$ such that
    \begin{equation*}
        \langle F, g\rangle_{\Lambda^{-s,A'}(\Omega)\times\Lambda^{s,A}_0(\Omega)}=\int_{\Omega}{f g}\,dx+\int_{\R^d}{\boldsymbol{f}\cdot D^s g}\,dx,\quad \forall g\in \Lambda^{s,A}_0(\Omega).
    \end{equation*}
\end{proposition}
\begin{proof}
    Let $P:\Lambda^{s,A}_0(\Omega)\to L^A(\Omega)\times L^A(\R^d;\R^d)$ which is defined by $P(g)=[g,D^s g]$. Since
    \begin{equation*}
        \|P(g)\|_{L^A(\R^d)\times L^A(\R^d;\R^d)}=\|g\|_{\Lambda^{s,A}_0(\Omega)}
    \end{equation*}
    then, $P$ maps isometrically $\Lambda^{s,A}_0(\Omega)$ onto a subspace $W\subset L^A(\R^d;\R^{d+1})$. Define $F^*$ on $W$ as $\langle F^*, P(g)\rangle=\langle F, g\rangle$, which satisfies $\|F^*\|_{W'}=\|F\|_{\Lambda^{-s,A'}(\Omega)}$. By the Hahn-Banach theorem, there exists an extension $\Tilde{F}$ of $F^*$ in $L^A(\Omega)\times L^A(\R^d;\R^d)$ such that $\|F^*\|_{W'}=\|\Tilde{F}\|_{L^A(\Omega)\times L^A(\R^d;\R^d)}$. By the Riesz representation theorem for generalized Orlicz spaces \cite[Theorem 3.4.6]{harjulehto2019orlicz}, we have that there exists $f\in L^{A'}(\Omega)$ and $\boldsymbol{f}\in L^{A'}(\R^d;\R^d)$ such that
    \begin{equation*}
        \langle \Tilde{F}, h\rangle_{L^{A'}\times L^A}=\int_{\Omega}{f h}\,dx+\int_{\R^d}{\boldsymbol{f}\cdot\boldsymbol{h}}\,dx,\quad \forall h=[h, \boldsymbol{h}]\in  L^A(\Omega)\times L^A(\R^d;\R^d).
    \end{equation*}
    Then, the result follows from
    \begin{multline*}
        \langle F, g\rangle_{\Lambda^{-s,A'}(\Omega)\times\Lambda^{s,A}_0(\Omega)}
        =\langle F^*, P(g)\rangle_{W'\times W}
        =\langle \Tilde{F}, P(g)\rangle_{L^{A'}\times L^A}\\
        =\int_{\Omega}{f g}\,dx+\int_{\R^d}{\boldsymbol{f}\cdot D^s g}\,dx,\quad \forall g\in \Lambda^{s,A}_0(\Omega)
    \end{multline*}
\end{proof}

\subsection{Continuous embeddings and continuity in $s$}
With the fractional Sobolev-Orlicz space defined, we now prove some important embedding results that are going to be crucial in the next sections.
Our first result is the Sobolev inequality for fractional generalized Sobolev-Orlicz spaces.

\begin{theorem}[Sobolev inequality]
    Let $p,q,r\geq 1$ such that $s\in(0,1)$ such that $\gamma:=\frac{s}{d}=\frac{1}{p}-\frac{1}{q}$ and $r\in(\gamma,\frac{1}{p}]$. Consider two $\Phi$-functions $A,B\in\Phi(\R^d)$ satisfying \hyperref[def:a0]{(A0)}, \hyperref[def:a1]{(A1)} and \hyperref[def:a2]{(A2)}. Assume also that $A$ is \inc{p} and \dec{1/r}, $B$ is \inc{q} and \dec{\frac{1}{r-\gamma}}, and that there exist two constants $c_1,c_2>0$ such that
    \begin{equation*}
        c_1 A^{-1}(x,t)\leq t^\gamma B^{-1}(x,t)\leq c_2 A^{-1}(x,t)\qquad\forall \ell>0 \mbox{ and a.e. } x\in\R^d.
    \end{equation*}
    Then, $\Lambda^{s,A}_0(\Omega)\subset L^B(\Omega)$ with the inequality
    \begin{equation*}
        \|u\|_{L^B(\Omega)}\leq C\|D^s u\|_{L^A(\R^d;\R^d)}
    \end{equation*}
\end{theorem}
\begin{proof}
    From \cite[Corollary 5.4.5]{harjulehto2019orlicz},
    \begin{equation*}
        \|I_s f\|_{L^B(\R^d)}\leq C\|f\|_{L^A(\R^d)},\quad\forall f\in L^A(\R^d).
    \end{equation*}
    Let $u\in C^\infty_c(\Omega)$. If we set $f=\mathcal{R}\cdot D^s u$, with $\mathcal{R}$ being the vectorial Riesz transform, then the second identity of \eqref{eq:fundamental_theorem_calculus} yields
    \begin{equation*}
        \|u\|_{L^B(\Omega)}\leq C\|\mathcal{R}\cdot D^s u\|_{L^A(\R^d)}.
    \end{equation*}
    Since the $\mathcal{R}$ is a Calderón-Zygmund singular integral operator, we can then apply \cite[Corollary 5.4.3]{harjulehto2019orlicz} to obtain
    \begin{equation*}
        \|\mathcal{R}\cdot D^s u\|_{L^A(\R^d)}\leq C\|D^s u\|_{L^A(\R^d;\R^d)}.
    \end{equation*}
\end{proof}

Another result that we prove here is the fractional Poincaré inequality for the fractional generalized Orlicz-Sobolev spaces.

\begin{theorem}[Poincaré inequality]\label{thm:poincare_inequality}
    Let $\Omega\subset\R^d$ be a bounded open set, $s\in(0,1)$, and $A\in\Phi(\R^d)$ satisfying \hyperref[def:a0]{(A0)}, \hyperref[def:a1]{(A1)} and \hyperref[def:a2]{(A2)}. Then, there exists a positive constant $C>0$, independent of $s$ and $u$, such that
    \begin{equation*}
        \|u\|_{L^A(\Omega)}\leq \frac{C}{1-2^{-s}}\|D^s u\|_{L^A(\Omega_1)},\quad \forall u\in \Lambda^{s,A}_0(\Omega).
    \end{equation*}
\end{theorem}
\begin{proof}
    Let $u\in C^\infty_c(\Omega)$. Consider also $R>1$ to be sufficiently large so that $\Omega\subset B_R(0)$. Making use of the first identity in \eqref{eq:fundamental_theorem_calculus},
    \begin{equation}\label{eq:split_norm_for_poincare}
        \begin{aligned}
            &\|u\|_{L^A(\Omega)}\leq \mu(d,-s)\left\|\left(\int_{B_{2R}}{\frac{|D^s u(y)|}{|x-y|^{d-s}}}\,dy+\int_{B^s_{2R}}{\frac{|D^s u(y)|}{|x-y|^{d-s}}}\,dy\right)\right\|_{L^A(\Omega)}\\
            &\leq \mu(d,-s)\left(\left\|\int_{B_{2R}}{\frac{|D^s u(y)|}{|x-y|^{d-s}}}\,dy\right\|_{L^A(\Omega)}+\left\|\int_{B^c_{2R}}{\frac{|D^s u(y)|}{|x-y|^{d-s}}}\,dy\right\|_{L^A(\Omega)}\right).
        \end{aligned}
    \end{equation}
    
    Just like in \cite{bellido2020gamma}, we proceed by analyzing the two norms on the right-hand-side of the previous inequality separately.
    
    Let $y\in B^c_{2R}(0)$. Then, from the integral definition of $D^s$ for regular functions
    \begin{align*}
        |D^s u(y)|\leq \mu(d,s)\int_\Omega{\frac{|u(z)|}{|y-z|^{d+s}}}\,dz.
    \end{align*}
    Observing that for $z\in\Omega\subset B_R(0)$ and $y\in B^c_{2R}(0)$, we have $1/2|y|\leq |y-z|$, then Young's inequality yields
    \begin{align*}
        |D^s u(y)|&\leq \mu(d,s)\int_\Omega{|u(z)|2^{d+s}|y|^{-d-s}}\,dz\leq \frac{2^{d+1}\mu(d,s)}{|y|^{d+s}}\int_{\Omega}{|u(z)|}\,dz\\
        & \leq \frac{2^{d+2}\mu(d,s)}{|y|^{d+s}}\|\chi_{\Omega}\|_{L^{A'}(\Omega)}\|u\|_{L^A(\Omega)}.
    \end{align*}
    This pointwise estimate allows us to obtain
    \begin{align*}
        &\int_{B^c_{2R}}{\frac{|D^s u(y)|}{|x-y|^{d-s}}}\,dy\\
        &\qquad\qquad\leq 2^{d+2}\mu(d,s)S_d\|\chi_{\Omega}\|_{L^{A'}(\Omega)}\|u\|_{L^A(\Omega)}\int_{B^c_{2R}}{\frac{1}{|y|^{d+s}}\frac{1}{|x-y|^{d-s}}}\,dy\\
        &\qquad\qquad\leq \frac{2^{3-d}\mu(d,s)}{d}S_dR^{-d}\|\chi_{\Omega}\|_{L^{A'}(\Omega)} \|u\|_{L^A(\Omega)},
    \end{align*}
    which then implies that
    \begin{equation}\label{eq:estimate_orlicz_norm_outside_ball}
        \begin{aligned}
            \left\|\int_{B^c_{2R}}{\frac{|D^s u(y)|}{|x-y|^{d-s}}}\,dy\right\|_{L^A(\Omega)}\leq \frac{2^{3-d}\mu(d,s)}{d}S_dR^{-d}\|u\|_{L^A(\Omega)}\|\chi_{\Omega}\|_{L^{A'}(\Omega)}\|\chi_{\Omega}\|_{L^{B}(\Omega)}.
        \end{aligned}
    \end{equation}

    For the other norm on the right-hand side of \eqref{eq:split_norm_for_poincare}, we make use of \cite[Lemma 6.1.4]{diening2017lebesgue}. In fact, this lemma allows us to state that for almost every $x\in\R^d$,
    \begin{align*}
        \int_{B_{2R}}{\frac{|D^s u(y)|}{|x-y|^{N-s}}}\,dy&\leq \int_{\{|x-y|<3R\}}{\frac{|D^s u(y)|}{|x-y|^{N-s}}}\,dy\\
        &\leq \frac{\pi^{n/2}}{\Gamma(n/2+1)} (3R)^s\sum_{k=0}^\infty{2^{-sk}T_{k+k_0}|D^s u|(x)},
    \end{align*}
    with $k_0\in \Z$ being such that $2^{-k_0-1}\leq 3R< 2^{-k_0}$, and 
    \begin{equation*}
        T_{k+k_0}|D^s u|(x)=\sum_{\substack{Q \text{ dyadic}\\ \text{diam}(Q)=2^{-k-k_0}}}{\frac{\chi_{Q}(x)}{|Q|}\int_{2Q}{|D^su(y)|\,dy}}.
    \end{equation*}
    Using the triangle inequality and the uniform boundedness of $T_{k+k_0}: L^A(\Omega)\to L^A(B_{3R})$ proved in \cite[Theorem 4.4.3]{harjulehto2019orlicz}, there is a positive constant $C>0$, independent of $s$ and $u$ such that
    \begin{equation}\label{eq:estimate_orlicz_norm_inside_ball}
        \begin{aligned}
            \left\|\int_{B_{2R}}{\frac{|D^s u(y)|}{|x-y|^{N-s}}}\,dy\right\|_{L^A(\Omega)}&
            \leq 2^d(3R)^s\sum_{k=0}^\infty{2^{-sk}\|T_{k+k_0}|D^s u|\|_{L^A(\Omega)}}\\
            &\leq CR^s\sum_{k=0}^\infty{2^{-sk}\|D^s u\|_{L^A(B_{3R})}}\\
            &\leq \frac{CR^s}{1-2^{-s}}\|D^s u\|_{L^A(B_{3R})}\leq \frac{CR}{1-2^{-s}}\|D^s u\|_{L^A(B_{3R})}.
        \end{aligned}
    \end{equation}

    Applying \eqref{eq:estimate_orlicz_norm_outside_ball} and \eqref{eq:estimate_orlicz_norm_inside_ball} into \eqref{eq:split_norm_for_poincare}, and using the fact that $\sup_{s\in[-1,1]}{\mu(N,s)}<\infty$, we deduce that there exists a single positive constant $C>0$, independent of $s$ and $u$, such that
    \begin{equation*}
        \|u\|_{L^A(\Omega)}\leq C\left(\frac{R}{1-2^{-s}}\|D^s u\|_{L^A(B_{3R})}+R^{-d}\|u\|_{L^A(\Omega)}\right).
    \end{equation*}
    Then, we just need to choose $R$ sufficiently large so that $CR^{-d}\leq 1/2$.
\end{proof}

\begin{remark}
    The extreme cases $s=1$ and $s=0$ are well known in the literature. For instance, the case $s=1$, corresponds to the classical Poincaré inequality when we make the identification $D^1=D$, \cite[Theorem 6.2.8]{harjulehto2019orlicz}, while the case $s=0$ corresponds to the fact that $D^0=-\mathcal{R}$ transforms continuously functions in $L^A(\R^d)$ to functions in $L^A(\R^d;\R^d)$, \cite[Corollary 5.4.3]{harjulehto2019orlicz}.
\end{remark}

\begin{remark}
    The Poincaré inequality obtained in Theorem \ref{thm:poincare_inequality} can be used to improve the characterization of the elements of $\Lambda^{-s,A'}(\Omega)$ obtained in Proposition \ref{prop:characterization_dual}. In fact, if besides the assumptions of Proposition \ref{prop:characterization_dual} we also assume that $\Omega$ is a bounded open set, then it is possible to prove that for each $F\in \Lambda^{-s,A'}(\Omega)$, there exists $\boldsymbol{f}\in L^{A'}(\R^d;\R^d)$ such that
    \begin{equation*}
        \langle F, g\rangle_{\Lambda^{-s,A'}(\Omega)\times\Lambda^{s,A}_0(\Omega)}=\int_{\R^d}{\boldsymbol{f}\cdot D^s g}\,dx,\quad \forall g\in \Lambda^{s,A}_0(\Omega).
    \end{equation*}
    This is a simple consequence of the fact that the Poincaré inequality allows us to state that $u\mapsto \|D^s u\|_{L^A(\R^d;\R^d)}$ is a norm for $\Lambda^{s,A}_0(\Omega)$.
\end{remark}

We can extend the previous method to obtain an order on the family of fractional generalized Orlicz-Sobolev spaces, $\Lambda^{s,A}_0(\Omega)$, with respect to the fractional parameter $s$.

\begin{theorem}\label{thm:spaces_decrease}
    Let $\Omega$ be a bounded open set, $0<\sigma<s\leq 1$, and $A\in\Phi(\R^d)$ satisfying \hyperref[def:a0]{(A0)}, \hyperref[def:a1]{(A1)} and \hyperref[def:a2]{(A2)}. There exists a constant $C>0$ independent of $s$ and $\sigma$ such that
    \begin{multline*}
        \|D^\sigma u\|_{L^A(\R^d)}\\
        \leq C\left(\frac{1}{d-1+\sigma}\left(1+\frac{1}{1-2^{-\sigma}}\right)+\frac{1}{\sigma}\|\chi_\Omega\|_{L^{A'}}\frac{1}{1-2^{-\sigma}}\right)\|D^s u\|_{L^A(\R^d)}
    \end{multline*}
\end{theorem}
\begin{proof}
    Let $u\in C^\infty_c(\Omega)$. Consider the set $\Omega'=\Omega+B_1$ and a constant $R>1$ sufficiently large so that $\Omega'\subset B_R(0)$. The idea is to prove that 
    \begin{equation*}
        \|D^\sigma u\|_{L^A(\R^d)}\leq \|D^\sigma u\|_{L^A(\Omega')}+\|D^\sigma u\|_{L^A((\Omega')^c)}\leq C\|D^s u\|_{L^A(\R^d)}.
    \end{equation*}
    We start by proving that $\|D^\sigma u\|_{L^A(\Omega')}\leq C\|D^s u\|_{L^A(\R^d)}$. From the semigroup property of the Riesz potential, we can write
    \begin{equation*}
        D^\sigma u=I_{\delta}D^s u
    \end{equation*}
    where $\delta=s-\sigma>0$. Following the same steps as in the proof of Theorem \ref{thm:poincare_inequality} and using the fact that $\frac{1}{1-2^{-\delta}}\leq \frac{2}{\delta}$ and \cite[Lemma 2.4]{bellido2020gamma} for the existence of a constant $C>0$ independent of $\delta$ such that $\mu(d,1-\delta)<C\delta$,
    \begin{align*}
        &\|D^\sigma u\|_{L^A(\Omega')}\\
        &\quad\leq \frac{\mu(d,1-\delta)}{d-\delta}\left(\left\|\int_{B_{2R}}{\frac{|D^s u(y)|}{|x-y|^{d-\delta}}}\,dy\right\|_{L^A(\Omega')}+\left\|\int_{B^c_{2R}}{\frac{|D^s u(y)|}{|x-y|^{d-\delta}}}\,dy\right\|_{L^A(\Omega')}\right)\\
        &\quad\leq C\frac{\mu(d,1-\delta)}{d-\delta}\left(\frac{R}{1-2^{-\delta}}\|D^s u\|_{L^A(\R^d)}+R^{-d-\sigma}\|u\|_{L^A(\Omega)}\right)\\
        &\quad\leq C\frac{\mu(d,1-\delta)}{d-\delta}\left(\frac{1}{1-2^{-\delta}}+\frac{1}{1-2^{-\sigma}}\right)\|D^s u\|_{L^A(\R^d)}\\
        &\quad\leq C\frac{\mu(d,1-\delta)}{d-1+\sigma}\left(\frac{2}{\delta}+\frac{1}{1-2^{-\sigma}}\right)\|D^s u\|_{L^A(\R^d)}\\
        &\quad\leq \frac{C}{d-1+\sigma}\left(1+\frac{1}{1-2^{-\sigma}}\right)\|D^s u\|_{L^A(\R^d)}.
    \end{align*}
    To prove that $\|D^\sigma u\|_{L^A((\Omega')^c)}\leq C\|D^s u\|_{L^A(\R^d)}$ holds, we first apply Jensen's inequality, Lemma \ref{lemma:relation_classical_and_orlicz}, and \cite[Lemma 3.7.7]{harjulehto2019orlicz}
    \begin{equation}\label{eq:estimate_poincare_complement_augmented_domain}
        \begin{aligned}
            &\|D^\sigma u\|_{L^A((\Omega')^c)}\\
            &\leq C\left\|\int_\Omega{\frac{|u(y)|}{|x-y|^{d+\sigma}}}\,dy\right\|_{L^A((\Omega')^c)}
            \leq C\int_{\Omega}{|u(y)|\left\|\frac{1}{|x-y|^{(d+\sigma)}}\right\|_{L^A((\Omega')^c)}}\,dy\\
            &\leq C\int_{\Omega}{|u(y)|\max\left\{\left\|\frac{1}{|x-y|^{(d+\sigma)}}\right\|_{L^p((\Omega')^c)},\left\|\frac{1}{|x-y|^{(d+\sigma)}}\right\|_{L^q((\Omega')^c)} \right\}}\,dy.
        \end{aligned}
    \end{equation}
    Then, using the fact that $(\Omega')^c-y\leq B_1^c$, we obtain that
    \begin{align*}
        &\max\left\{\left\|\frac{1}{|x-y|^{(d+\sigma)}}\right\|_{L^p((\Omega')^c)},\left\|\frac{1}{|x-y|^{(d+\sigma)}}\right\|_{L^q((\Omega')^c)} \right\}\\
        &\qquad\qquad\leq \max\left\{\left\|\frac{1}{|x|^{(d+\sigma)}}\right\|_{L^p(B_1^c)},\left\|\frac{1}{|x|^{(d+\sigma)}}\right\|_{L^q(B_1^c)} \right\}\\
        &\qquad\qquad\leq \max\left\{\left(\frac{C}{(d+\sigma)p-d}\right)^{1/p},\left(\frac{C}{(d+\sigma)q-d}\right)^{1/q} \right\}\\
        &\qquad\qquad\leq \max\left\{\left(\frac{C}{\sigma}\right)^{1/p},\left(\frac{C}{\sigma}\right)^{1/q} \right\}\leq \max\left\{1,\frac{C}{\sigma}\right\}.
    \end{align*}
    Applying this estimate to \eqref{eq:estimate_poincare_complement_augmented_domain}, yields
    \begin{align*}
        \|D^\sigma u\|_{L^A(\Omega_C^c)}&\leq \max\{1,\frac{C}{\sigma}\}\int_\Omega{|u(y)|}\,dy\leq 2\max\{1,\frac{C}{\sigma}\}\|\chi_\Omega\|_{L^{A'}}\|u\|_{L^A(\Omega)}\\
        &\leq \frac{C}{\sigma}\|\chi_\Omega\|_{L^{A'}}\frac{1}{1-2^{-\sigma}}\|D^s u\|_{L^A(\R^d)}.
    \end{align*}
\end{proof}

We also present an interpolation inequality, of the Gagliardo-Nirenberg type, which generalizes \cite[Theorem 4.5]{brue2022AsymptoticsII} for $\Lambda^{s,A}_0(\Omega)$. To prove such inequality, we make use of the following lemma:

\begin{lemma}\label{lemma:mihklin_hormander_for_orlicz}
    Let $0\leq \sigma\leq s\leq 1$ and $A\in\Phi(\R^d)$ satisfying \hyperref[def:a0]{(A0)}, \hyperref[def:a1]{(A1)}, \hyperref[def:a2]{(A2)}, \inc{p} and \dec{q}. Consider the Fourier symbol $m_{s,\sigma}(\xi)=\frac{|\xi|^\sigma}{1+|\xi|^s}$ and let us denote by $T_{s,\sigma}\varphi= \varphi*\mathcal{F}^{-1}(m_{s,\sigma})$ for $\varphi\in\mathcal{S}(\R^d)$ the operator associated to the symbol. Then,
    \begin{equation*}
        \|T_{s,\sigma} v\|_{L^A(\R^d)}\leq C\|v\|_{L^A(\R^d)},\quad\forall v\in L^A(\R^d).
    \end{equation*}
\end{lemma}
\begin{proof}
    If we prove that $T_{s,\sigma}$ is a Calderón-Zygmund singular integral operator, then one just needs to apply \cite[Corollary 5.4.3]{harjulehto2019orlicz} to prove the lemma.

    In order to prove that $T_{s,\sigma}$ is a Calderón-Zygmund singular integral operator, we start by recalling the Faà di Bruno's formula for multivariable functions, \cite{hardy2006Combinatorics}. Let $f:\mathrm{dom}(f)\subset\R\to\R$ and $g:\mathrm{dom}(g)\subset\R^d\to\mathrm{dom}(f)$ both $N$-times continuously differentiable, then
    \begin{equation}\label{eq:faa_di_bruno_formula}
        \partial^\alpha(f\circ g)(\xi)=\sum_{\pi\in\Pi_\alpha}{f^{(|\pi|)}(g(\xi))\prod_{\beta\in\pi}{\partial^\beta g(\xi)}}
    \end{equation}
    where $\Pi_\alpha$ is the set of all partitions of the multiindex $\alpha$, and $\beta\in \pi$ runs over all the multiindices that form the partition $\pi$. For our particular case, we set $f:\R^+\to\R$ and $g:\R^d\setminus\{0\}\to\R^+$ to be defined as
    \begin{equation*}
        f(y)=\frac{y^\sigma}{1+y^s},\quad \mbox{ and }\quad g(\xi)=|\xi|.
    \end{equation*}
    In order to compute \eqref{eq:faa_di_bruno_formula}, first we need to study the derivatives of $f$ and $g$.

    For $g$, it is well-known that for each multiindex $\beta$ there exists a constant $C{|\beta|}>0$ such that
    \begin{equation*}
        |\partial^\beta g(\xi)|\leq C_{|\beta|}|\xi|^{1-|\beta|},\quad\forall \xi\in\R^d\setminus\{0\}.
    \end{equation*}
    Hence, there exits a constant $C_\pi>0$ dependent on the partition $\pi$ such that
    \begin{equation}\label{eq:estimate_for_faa_product}
        \left|\prod_{\beta\in\pi}{\partial^\beta g(\xi)}\right|\leq C_{\pi}|\xi|^{|\pi|-|\alpha|}.
    \end{equation}

    For the derivatives of $f$ we observe that from the generalized product rule,
    \begin{equation*}
        (y^\sigma)^{(|\pi|)}=\left(\frac{y^\sigma}{1+y^s}(1+y^s)\right)^{(|\pi|)}=\sum_{j=0}^{|\pi|}{{|\pi| \choose j} \left(\frac{y^\sigma}{1+y^s}\right)^{(|\pi|-j)}(1+y^s)^{(j)}},
    \end{equation*}
    and hence
    \begin{equation*}
        \left(\frac{y^\sigma}{1+y^s}\right)^{(|\pi|)}=\frac{1}{1+y^s}\left((y^\sigma)^{(|\pi|}-\sum_{j=1}^{|\pi|}{{|\pi|\choose j}\left(\frac{y^\sigma}{1+y^s}\right)^{(|\pi|-j)}(1+y^s)^{(j)}}\right).
    \end{equation*}
    This recursive formula for the higher derivative of the quotient allows us to obtain estimates for $\left(\frac{y^\sigma}{1+y^s}\right)^{(|\pi|)}$ through induction. In fact, observe that for $y\in\R^+$ and $k=0$, we have 
    \begin{equation*}
        \left(\frac{y^\sigma}{1+y^s}\right)^{(0)}\leq 2.
    \end{equation*}
    Let us then suppose that for all $k<n$ there exists a constant $C_k>0$ independent of $s$ and $\sigma$, but dependent of $k$ such that 
    \begin{equation*}
        \left(\frac{y^\sigma}{1+y^s}\right)^{(k)}\leq C_k y^{-k}.
    \end{equation*}
    Then,
    \begin{align*}
        &\left(\frac{y^\sigma}{1+y^s}\right)^{(n)}=\frac{1}{1+y^s}\left((y^\sigma)^{(n)}-\sum_{j=1}^{n}{{n\choose j}\left(\frac{y^\sigma}{1+y^s}\right)^{(n-j)}(1+y^s)^{(j)}}\right)\\
        &\qquad\leq \frac{1}{1+y^s}\left(\prod_{p=0}^{n-1}{(\sigma-p)}y^{\sigma-n}-\sum_{j=1}^{n}{{n\choose j}C_{n-j}y^{j-n}\prod_{q=0}^{j-1}{(s-q)}y^{s-j}}\right)\\
        &\qquad\leq \left(\prod_{p=0}^{n-1}{(\sigma-p)}y^{-n}-\sum_{j=1}^{n}{{n\choose j}C_{n-j}y^{j-n}\prod_{q=0}^{j-1}{(s-q)}y^{-j}}\right)\\
        &\qquad\leq \left(\left(\sup_{\sigma\in[0,1]}{\prod_{p=0}^{n-1}{(\sigma-p)}}\right)+\sum_{j=1}^{n}{{n\choose j}C_{n-j}\left(\sup_{s\in[0,1]}{\prod_{q=0}^{j-1}{(s-q)}}\right)}\right)y^{-n}\\
        &\qquad = C_n y^{-n}.
    \end{align*}
    Hence for every partition $\pi$,
    \begin{equation}\label{eq:estimate_for_faa_quotient}
        \left(\frac{y^\sigma}{1+y^s}\right)^{(|\pi|)}\leq C_{|\pi|} y^{-|\pi|}.
    \end{equation}
    Applying both \eqref{eq:estimate_for_faa_product} and \eqref{eq:estimate_for_faa_quotient} to \eqref{eq:faa_di_bruno_formula}, we obtain that every multiindex $\alpha$,
    \begin{equation*}
        |\partial^\alpha m_{s,\sigma}(\xi)|\leq \sum_{\pi\in\Pi_\alpha}{C_{\pi}|\xi|^{-|\pi|}|\xi|^{|\pi|-|\alpha|}}=C_\alpha |\xi|^{-|\alpha|}, \quad\forall \xi\in\R^d\setminus\{0\}.
    \end{equation*}
    As a simple consequence of this estimates, we can use \cite[Proposition 4.27]{abels2012pseudodifferential} to obtain a kernel $k_{s,\sigma}\in C^1(\R^d\setminus\{0\})$ to the operator $T_{s,\sigma}$, i.e.,
    \begin{equation*}
        T_{s,\sigma}\varphi(x)=\int_{\R^d}{k_{s,\sigma}(x-y)\varphi(y)}\,dy, \quad \mbox{ for all } x\notin \mathrm{supp}{(\varphi)}, \mbox{ with } \varphi\in\mathcal{S}(\R^d), 
    \end{equation*}
    satisfying the estimates
    \begin{equation}\label{eq:partial_derivatives_kernel}
        |\partial^\gamma k_{s,\sigma}(z)|\leq C|z|^{-d-|\gamma|} \quad \mbox{ for all } z\neq 0 \mbox{ and } |\gamma|\leq 1.
    \end{equation}
    By defining $K_{s,\sigma}(x,y)=k_{s,\sigma}(x-y)$, we can observe that \eqref{eq:estimate_function_strong_mihklin} follows immediatly from \eqref{eq:partial_derivatives_kernel} with $|\gamma|=0$, while \eqref{eq:estimate_derivative_strong_mihklin} follows from
    the fact that for $|h|\leq \frac{1}{2}|x-y|$ and $t\in[-1,1]$ we have $|x-y-th|\geq \frac{1}{2}|x-y|$ and consequently
    \begin{align*}
        &|k_{s,\sigma}(x-y)-k_{s,\sigma}(x-y-h)|+|k_{s,\sigma}(x-y)-k_{s,\sigma}(x+h-y)|\\
        &\qquad\qquad\leq \sup_{\theta\in[0,1]}{|Dk_{s,\sigma}(x-y-\theta h)|}|h|+\sup_{\theta\in[0,1]}{|Dk_{s,\sigma}(x-y+\theta h)|}|h|\\
        &\qquad\qquad\leq C\frac{|h|}{|x-y|^{d+1}}.
    \end{align*}
\end{proof}

\begin{proposition}[Interpolation inequality]
    Let $A\in\Phi(\R^d)$ satisfying \hyperref[def:a0]{(A0)}, \hyperref[def:a1]{(A1)}, \hyperref[def:a2]{(A2)}, \inc{p} and \dec{q} with $1<p\leq q<\infty$. For every $0\leq r\leq s\leq t\leq 1$ and every $u\in \Lambda_0^{t,A}(\Omega)$, 
    \begin{equation}\label{eq:interpolation_inequality}
        \|D^s u\|_{L^A(\R^d)}\leq C\|D^r u\|^{\frac{t-s}{t-r}}_{L^A(\R^d)}\|D^{t} u\|^{\frac{s-r}{t-r}}_{L^A(\R^d)}.
    \end{equation}
    And in particular, when $r=0$, we have
    \begin{equation}\label{eq:interpolation_inequality_with_0}
        \|D^s u\|_{L^A(\R^d)}\leq C\|u\|^{\frac{t-s}{t}}_{L^A(\R^d)}\|D^{t} u\|^{\frac{s}{t}}_{L^A(\R^d)}.
    \end{equation}
\end{proposition}
\begin{proof}
    This proof is inspired by the proof of \cite[Theorem 4.5]{brue2022AsymptoticsII}. Since $C^\infty_c(\Omega)$ is dense in $\Lambda^{s,A}_0(\Omega)$, we can assume that $u\in C^\infty_c(\Omega)$. From the definition of $T_{t,s}$ in terms of the Fourier symbol $m_{t,s}$, we have that $(-\Delta)^{s/2}f=T_{t,s}\circ(Id+(-\Delta)^{t/2})$. Hence, there exists a constant $C>0$ independent of $s$ such that
    \begin{align*}
        \|(-\Delta)^{s/2}u\|_{L^A(\R^d)}&=\|T_{t,s}\circ(Id+(-\Delta)^{t/2})u\|_{L^A(\R^d)}\\
        &\leq C\|u+(-\Delta)^{t/2}u\|_{L^A(\R^d)}\\
        &\leq C(\|u\|_{L^A(\R^d)}+\|(-\Delta)^{t/2}u\|_{L^A(\R^d)})\\
        &\leq C\|u\|^{\frac{t-s}{t}}_{L^A(\R^d)}\|(-\Delta)^{t/2}u\|^{\frac{s}{t}}_{L^A(\R^d)},
    \end{align*}
    where we used Lemma \eqref{lemma:mihklin_hormander_for_orlicz} for the first inequality, and an argument where we optimize the dilation of $u$ for the last inequality.
    From \eqref{eq:fundamental_theorem_calculus} and the fact that the Riesz transform satisfies $\mathcal{R}\cdot\mathcal{R}=-\mathrm{Id}$ by \cite[Proposition 5.1.16]{grafakos2014classical}, then $D^s u=\mathcal{R}(-\Delta)^{s/2}u$. Moreover, because $\mathcal{R}$ is a Calderón-Zygmund singular integral operator, then \cite[Corollary 5.4.3]{harjulehto2019orlicz} allows us to conclude that
    \begin{equation*}
        \|D^\sigma u\|_{L^A(\R^d)}\leq C\|\mathcal{R}u\|^{\frac{s-\sigma}{s}}_{L^A(\R^d)}\|D^s u\|^{\frac{s}{\sigma}}_{L^A(\R^d)}\leq C\|u\|^{\frac{s-\sigma}{s}}_{L^A(\R^d)}\|D^s u\|^{\frac{s}{\sigma}}_{L^A(\R^d)}.
    \end{equation*}
    The inequality \eqref{eq:interpolation_inequality} can be obtained in a similar way, if we also use the identity $(-\Delta)^{t/2}D^r u=\mathcal{R}(-\Delta)^{\frac{r+t}{2}}u$ for test functions $u\in C^\infty_c(\Omega)$. Compare with the proof of \cite[Theorem 4.5]{brue2022AsymptoticsII}.
\end{proof}

As a consequence of this interpolation inequality, we prove two approximation results for the fractional gradient $D^s u$ by other fractional gradients. The first one regards the continuity of the application that maps $s\in [0,1]$ into $D^s u$ for some \textit{a priori} fixed sufficiently regular function $u$, while the second one is concerned with the existence of a sequence $\{u_n\}$ such that $D^{s_n}u_n$ converges to $D^\sigma u$ in $L^A$ for any given $u\in\Lambda^{\sigma,A}_0(\Omega)$ and $s_n\to\sigma$.

\begin{proposition}\label{prop:conv_frac_grad_fixed_function_stationary}
    Let $A\in\Phi(\R^d)$ satisfying \hyperref[def:a0]{(A0)}, \hyperref[def:a1]{(A1)}, \hyperref[def:a2]{(A2)}, \inc{p} and \dec{q} with $1<p\leq q<\infty$. Let us also consider a sequence $\{s_n\}\subset [0,1]$ such that $s_n\to\sigma\in[0,1]$ and $\overline{s}=\sup\{s_n\}$. If $u\in \Lambda^{\overline{s},A}_0(\Omega)$, then $D^{s_n} u\to D^\sigma u$ in $L^A(\R^d;\R^d)$.
\end{proposition}
\begin{proof}
    The result was proven in \cite[Theorem C.1]{brue2022AsymptoticsII} for test functions in the case in which $A(t)=\frac{1}{p}|t|^p$. For the the more general class of functions that we are considering, the idea is to approximate $u$ by suitable test functions $\{u_k\}\subset C^\infty_c(\Omega)$ where we can use Lemma \ref{lemma:relation_classical_and_orlicz} to obtain convergence in $L^A$ from the convergence in $L^p$ and $L^q$.
    
    In fact, for the construction of the sequence of test functions, we observe that from the definition of $\Lambda^{\overline{s},A}_0(\Omega)$, there exists a sequence of test functions $\{u_k\}\subset C^\infty_c(\Omega)$ such that 
    \begin{equation*}
        \|u_k-u\|_{L^A(\R^d)}+\|D^{\overline{s}} u_k-D^{\overline{s}} u\|_{L^A(\R^d;\R^d)}\leq \frac{1}{2k}.
    \end{equation*}
    From triangle inequality and the interpolation inequality \eqref{eq:interpolation_inequality}, 
    \begin{align*}
        &\|D^{s_n} u-D^\sigma u\|_{L^A(\R^d;\R^d)}\\
        &\quad\leq \|D^{s_n} u-D^{s_n} u_k\|_{L^A(\R^d;\R^d)}+\|D^{s_n} u_k-D^\sigma u_k\|_{L^A(\R^d;\R^d)}\\
        &\qquad\quad+\|D^\sigma u_k-D^\sigma u\|_{L^A(\R^d;\R^d)}\\
        &\quad\leq C\|u-u_k\|^{\frac{\overline{s}-s_n}{\overline{s}}}_{L^A(\R^d;\R^d)}\|D^{\overline{s}} u-D^{\overline{s}} u_k\|^{\frac{s_n}{\overline{s}}}_{L^A(\R^d;\R^d)}+\|D^{s_n} u_k-D^\sigma u_k\|_{L^A(\R^d;\R^d)}\\
        &\qquad\quad+C\|u-u_k\|^{\frac{\overline{s}-\sigma}{\overline{s}}}_{L^A(\R^d;\R^d)}\|D^{\overline{s}} u-D^{\overline{s}} u_k\|^{\frac{\sigma}{\overline{s}}}_{L^A(\R^d;\R^d)}\\
        &\quad\leq \frac{C}{k}+\|D^{s_n} u_k-D^\sigma u_k\|_{L^A(\R^d;\R^d)}
    \end{align*}
    which is possible because $0\leq s_n,\sigma\leq \overline{s}$. As mentioned before, from Lemma \ref{lemma:relation_classical_and_orlicz} and \cite[Theorem C.1]{brue2022AsymptoticsII},
    \begin{multline*}
        \lim_{n\to\infty}{\|D^{s_n} u_k-D^\sigma u_k\|_{L^A(\R^d)}}\\
        \leq C\lim_{n\to\infty}{(\|D^{s_m} u_m-D^\sigma u_m\|_{L^p(\R^d)}+\|D^{s_m} u_m-D^\sigma u_m\|_{L^q(\R^d)})}=0
    \end{multline*}
    for all $k\in\N$. Hence,
    \begin{align*}
        &\lim_{n\to\infty}{\|D^{s_n} f-D^\sigma f\|_{L^A(\R^d;\R^d)}}\leq  \frac{C}{k}
    \end{align*}
    which concludes the proof.
\end{proof}

\begin{proposition}\label{prop:existence_approximating_function}
    Let $0\leq \sigma\leq 1$ and let $\{s_n\}\subset[0,1]$ be such that $s_n\to \sigma$. For any $u\in \Lambda^{\sigma,A}_0(\Omega)$, there exists a sequence $\{u_n\}$ with $u_n\in \Lambda^{s_n,A}_0(\Omega)$, such that $u_n\to u$ in $L^A(\Omega)$ and $D^{s_n}u_n\to D^\sigma u$ in $L^A(\R^d;\R^d)$.
\end{proposition}
\begin{proof}
    Let $u\in \Lambda^{\sigma,A}_0(\Omega)$. By definition of $\Lambda^{\sigma,A}_0(\Omega)$ there exist a sequence of test functions $\{\varphi_j\}\subset C^\infty_c(\Omega)$ such that $\varphi_j\to u$ in $\Lambda^{\sigma,A}_0(\Omega)$ as $j\to\infty$. For each $m\in\N$, there exists $q_1(m)\in\N$ such that $j\geq q_1(m)$, for which
    \begin{equation*}
        \|\varphi_j-u\|_{L^A(\Omega)}+\|D^\sigma \varphi_j- D^\sigma u\|_{L^A(\R^d;\R^d)}<\frac{1}{2m}
    \end{equation*}
    For each $j\geq q_1(m)$, by Proposition \ref{prop:conv_frac_grad_fixed_function_stationary} there exists a positive integer $q_2(m,j)\geq j$ such that for all $k\geq q_2(m,j)$ and
    \begin{equation*}
        \|D^{s_k} \varphi_j- D^\sigma \varphi_j\|_{L^p(\R^d;\R^d)}\leq\frac{1}{2m}.
    \end{equation*}
    By defining the sequence
    \begin{equation*}
        u_n=\begin{cases}
            \varphi_1 & \mbox{ if } n< q(1)\\
            \varphi_{q(k)} & \mbox{ if } q(k)\leq n<q(k+1), \quad \mbox{ for } k\in \N
        \end{cases}
    \end{equation*}
    where $q(k)=q_2(k,q_1(k))$, we have that for $n\geq q_2(m)\geq q_1(m)$,
    \begin{align*}
        &\|u_n-u\|_{L^A(\Omega)}+\|D^{s_n} u_n- D^\sigma u\|_{L^A(\R^d;\R^d)}\\
        &\quad\leq \|u_n-u\|_{L^A(\Omega)}+\|D^{s_n} u_n-D^\sigma u_n\|_{L^A(\R^d;\R^d)}+\|D^\sigma u_n- D^\sigma u\|_{L^A(\R^d;\R^d)}\\
        &\quad\leq 1/m,
    \end{align*}
    which implies that $u_n\to u$ in $L^A(\Omega)$ and $D^{s_n} u_n\to D^\sigma u$ in $L^A(\R^d;\R^d)$.
\end{proof}

\subsection{Compactness results}
Regarding the compactness results that are available to the fractional generalized Sobolev-Orlicz spaces, we mention and prove two. The first is a generalization of the Rellich-Kondrachov's theorem based on \cite[Theorem 6.3.7]{harjulehto2019orlicz}.
\begin{theorem}[Compact Sobolev Embedding]\label{thm:compact_sobolev_embedding}
    Let $\Omega\subset\R^d$ be an open bounded set and $0<s<1$. Let $A\in\Phi(\R^d)$ satisfy \hyperref[def:a0]{(A0)}, \hyperref[def:a1]{(A1)}, \hyperref[def:a2]{(A2)}, \inc{p} and \dec{q} for some $1< p<q<\infty$. Then
    \begin{equation*}
        \Lambda^{s,A}_0(\Omega)\Subset L^A(\Omega)
    \end{equation*}
    where $\Subset$ stands for compact inclusion.
\end{theorem}
\begin{proof}
    Let $u_m\rightharpoonup u$ in $\Lambda^{s,A}_0(\Omega)$. Without loss of generality we can assume that $u=0$. Because by definition $C^\infty_c(\Omega)$ is dense in $\Lambda^{s,A}_0(\Omega)$, then by extension $u_m=I_s v_m$ where $v_m=(-\Delta)^{s/2} u_m$. Moreover, from the fractional fundamental theorem of calculus \eqref{eq:fundamental_theorem_calculus} and from the boundedness of the Riesz transform in $L^A(\R^d)$ \cite[Corollary 5.4.3]{harjulehto2019orlicz}, the norm $\|v_m\|_{L^A(\R^d)}$ is comparable with the norm of $\|D^s u_m\|_{L^A(\R^d)}$. Let us also define $u^\varepsilon_m=\eta_\varepsilon* u_m$ where $\eta\in C^{\infty}_c(B_1)$ and $\eta_\varepsilon(x)=\frac{1}{\varepsilon^n}\eta(\frac{x}{\varepsilon})$. We study the behaviour of
    \begin{equation*}
        \|u_m\|_{L^A(\Omega)}\leq \|u^\varepsilon_m-u_m\|_{L^A(\R^d)}+\|u^\varepsilon_m\|_{L^A(\R^d)}. 
    \end{equation*}
    
    Just like in the proof of \cite[Theorem 2.2]{shieh2018on},
    \begin{equation*}
        |u^\varepsilon_m(x)-u_m(x)|=\varepsilon^s\int_{B_1}{\int_{\R^d}{\eta(y)\left|\frac{1}{|z-y|^{d-s}}-\frac{1}{|z|^{d-s}}\right||v_m(x-\varepsilon z)|}\,dz}\,dy.
    \end{equation*}
    From the properties of the norm of $L^A(\R^d)$,
    \begin{align*}
        &\|u^\varepsilon_m(x)-u_m(x)\|_{L^A(\R^d)}\\
        &\leq\varepsilon^s \int_{B(0,1)}{\eta(y)\int_{\R^d}{\left|\frac{1}{|z-y|^{d-s}}-\frac{1}{|z|^{d-s}}\right|\|v_m(\cdot-\varepsilon z)\|_{L^A(\R^d)}}\,dz}\,dy\\
        &\leq\varepsilon^s \int_{B(0,1)}{\eta(y)\int_{\R^d}{\left|\frac{1}{|z-y|^{d-s}}-\frac{1}{|z|^{d-s}}\right|}\,dz}\,dy\|v_m\|_{L^A(\R^d)}.
    \end{align*}
    Once again referencing to the proof of \cite[Theorem 2.2]{shieh2018on},
    \begin{equation*}
        \sup_{y\in B(0,1)}{\int_{\R^d}{\left|\frac{1}{|z-y|^{d-s}}-\frac{1}{|z|^{d-s}}\right|}\,dz}<\infty,
    \end{equation*}
    so there exists a positive constant $C>0$ independent of $\varepsilon$ and $m$ such that
    \begin{align*}
        &\|u^\varepsilon_m(x)-u_m(x)\|_{L^A(\R^d)}\\
        &\leq \varepsilon^s\left( \sup_{y\in B(0,1)}{\int_{\R^d}{\left|\frac{1}{|z-y|^{d-s}}-\frac{1}{|z|^{d-s}}\right|}\,dz}\right)\left(\int_{B(0,1)}{\eta(y)}\,dy\right)\|v_m\|_{L^A(\R^d)}\\
        &\leq C\varepsilon^s\|D^s u_m\|_{L^A(\R^d)}.
    \end{align*}

    On the other hand, since $u_m\rightharpoonup 0$ in $\Lambda^{s,A}_0(\Omega)$, then for any fixed $\varepsilon>0$, we have by definition of convolution $u^\varepsilon_m(x)\to 0$ as $m\to \infty$. Moreover since $|\eta_\varepsilon|\leq C\varepsilon^{-d}$ and $\mathrm{supp}{\,u^\varepsilon_m}\subset \Omega_\varepsilon=\{x\in \R^d:\, \mathrm{dist}(x,\Omega)\leq \varepsilon\}$, then by Young's and Poincaré's inequalities, as well as the uniform boundedness of $D^s u_m$ in $L^A(\R^d)$,
    \begin{multline*}
        |u^\varepsilon_m(x)|\leq 2\|u_m\|_{L^A(\Omega)}\|\eta_\varepsilon(x-\cdot)\|_{L^{A'}(\Omega)}\\
        \leq C(s,\Omega)\varepsilon^{-n} \|D^s u_m\|_{L^A(\Omega_1)} \|\chi_{\Omega_\varepsilon}\|_{L^{A'}(\Omega)}\leq C(s,\Omega)\varepsilon^{-n}.
    \end{multline*}
    From the dominated convergence theorem, this implies that $\|u^\varepsilon_m\|_{L^A(\Omega)}\to 0$ as $m\to \infty$, for any $s>0$ fixed. Consequently,
    \begin{align*}
        &\limsup_{m\to\infty}{\|u_m\|_{L^A(\Omega)}}
        =\lim_{\varepsilon}{\limsup_{m\to\infty}{\|u_m\|_{L^A(\Omega)}}}\\
        &\qquad\quad\leq\lim_{\varepsilon}{\left( \varepsilon^s\limsup_{m\to\infty}{\|D^su_m\|_{L^A(\R^d)}}+\lim_{m\to\infty}{\|u^\varepsilon_m\|_{L^A(\Omega_1)}}\right)}=C\lim_{\varepsilon\to 0}{\varepsilon^s}=0. 
    \end{align*}
\end{proof}

The second and last compactness result of this subsection is a generalization of \cite[Theorem 4.2]{bellido2020gamma} to the fractional generalized Orlicz-Sobolev framework. This result has important consequences, as one can see from Section \ref{sec:dirichlet_problem}, to the study of the continuous dependence of the solution of \eqref{eq:dirichlet_problem} with respect to the fractional parameter $s$.

\begin{theorem}\label{thm:convergence_bellido}
    Let $\Omega\subset\R^d$ be an open bounded set and $A\in\Phi(\R^d)$ satisfying \hyperref[def:a0]{(A0)}, \hyperref[def:a1]{(A1)}, \hyperref[def:a2]{(A2)}, \inc{p} and \dec{q} for some $1<p\leq q<\infty$. Consider a sequence $\{s_n\}\subset (0,1]$ with $s_n\to \sigma$ as $n\to\infty$. If for a sequence of functions $\{u_n\}$, with $u_n\in \Lambda^{s_n,A}_0(\Omega)$ for each $n$, we have $\|D^{s_n} u_{s_n}\|_{L^A(\R^d)}\leq C$ with $C$ independent of $n$, then we can extract a subsequence which we still denote by $u_n$ such that
    \begin{equation*}
        u_n\to u \mbox{ in } L^A(\Omega) \quad \mbox{ and } \quad D^{s_n} u_n\rightharpoonup D u\mbox{ in } L^A(\R^d),
    \end{equation*}
    for some function $u\in \Lambda^{\sigma,A}_0(\Omega)$.
\end{theorem}
\begin{proof}
    From Theorems \ref{thm:spaces_decrease} and \ref{thm:poincare_inequality}, $u_s$ is uniformly bounded in $\Lambda^{s_*,A}_0(\Omega)$ for some $0<s_*<1$ satisfying $s_*<s_n$ for each $n\in\N$. From Banach-Alaoglu's theorem and Theorem \ref{thm:compact_sobolev_embedding}, there exist a subsequence which we still denote by $u_n$ such that
    \begin{equation*}
        u_n\to u \mbox{ in } L^A(\Omega) \quad \mbox{ and } \quad D^{s_n} u_n\rightharpoonup v\mbox{ in } L^A(\R^d),
    \end{equation*}
    for some functions $u\in L^A(\Omega)$ and $v\in L^A(\R^d)$. Using the duality between $D^s$ and $\mathrm{div}_s$,  we have
    \begin{equation*}
        \int_{\R^d}{D^{s_n}u_n\cdot \varphi}\,dx=-\int_{\R^d}{u_n\mathrm{div}_{s_n}\varphi}\,dx, \quad \forall \varphi\in C^\infty_c(\R^d;\R^d).
    \end{equation*}
    Taking the limit, we obtain
    \begin{equation*}
        \int_{\R^d}{v\cdot \varphi}\,dx=-\int_{\R^d}{u\mathrm{div}_{s_n}\varphi}\,dx, \quad \forall \varphi\in C^\infty_c(\R^d;\R^d),
    \end{equation*}
    which concludes that $v=D^\sigma u$ because test functions are dense in $L^A(\R^d;\R^d)$, see \cite[Theorem 3.7.15]{harjulehto2019orlicz}, and consequently $u\in \Lambda^{\sigma,A}_0(\Omega)$. 
\end{proof}

\section{Application to the Dirichlet problem}\label{sec:dirichlet_problem}
As an application of the functional framework that was developed in the previous section, we are able to study the existence and uniqueness of solutions to
\begin{equation}\label{eq:dirichlet_problem_for_applications}
    \begin{cases}
        -D^s\cdot(a(x,D^s u)D^s u)=F_s & \mbox{ in }\Omega\\
        u=0 & \mbox{ on } \R^d\setminus \Omega,
    \end{cases}
\end{equation}
as well as their continuous dependence with respect to the fractional parameter $s$. In fact, for the existence and uniqueness we have the following result.
\begin{theorem}\label{thm:exitence_solutions}
    Let $\Omega$ be a bounded open set, $0\leq s\leq 1$, and $a$ be a measurable and bounded function in $x$ for all $r>0$ and Lipschitz continuous in $r$ for a.e. $x\in\R^d$, satisfying \eqref{eq:hypothesis_on_a} with some constants $a_-,a_+>0$. Assume also that the $\Phi$-function $A$ given by \eqref{eq:relation_between_a_and_A} also satisfies \hyperref[def:a0]{(A0)}, \hyperref[def:a1]{(A1)} and \hyperref[def:a2]{(A2)}. and that the function $F_s\in\Lambda^{-s,A'}(\Omega)$. Then, there exists a unique solution to \eqref{eq:dirichlet_problem_for_applications}. 
\end{theorem}
\begin{proof}
    The existence of the solution is a simple consequence of the direct method of Calculus of Variations, since we have that $A(x,\cdot)$ is convex for a.e. $x\in\R^d$, and it satisfies 
    \begin{equation*}
        (1+a_-)A(x,\ell)\leq a(x,\ell)\ell^2\leq (1+a_+)A(x,\ell),
    \end{equation*}
    as well as the Poincaré inequality, Theorem \eqref{thm:poincare_inequality}.

    For the uniqueness, we just need to prove that  $-\Delta^s_A$ is strictly monotone. We prove the strict monotonicity using the same argument as in \cite[Theorem 3.1]{lo2024orlicz}. Indeed, by defining $\theta_t=t D^s u+(1-t)D^s v$ for any two functions $u,v,\in \Lambda^{s,A}_0(\Omega)$, and 
    \begin{equation*}
        J(x)=\int_0^1{a(x,|\theta_t|)+|\theta_t|\partial_r a(x,|\theta_t|)}\,dt>0
    \end{equation*}
    where we used \eqref{eq:hypothesis_on_a}, we obtain that
    \begin{align*}
        \langle -\Delta^s_A u-\Delta^s_A v, u-v\rangle
        &=\int_{\R^d}{(a(x,|D^s u|)D^s u-a(x, |D^s v|)D^s v)\cdot (D^su -D^s v)}\,dx\\
        &=\int_{\R^d}{J(x)|D^s u-D^s v|^2}\,dx\geq 0
    \end{align*}
    with equality only when $u=v$.
\end{proof}

For the continuous dependence of the solutions of \eqref{eq:dirichlet_problem_for_applications} with respect to the fractional parameter $s$, we need the following lemmas.

\begin{lemma}\label{lemma:estimate_conjugate_derivative}
    Let $a(x,r)$ be a measurable and bounded function in $x$ for all $r>0$ and Lipschitz continuous in $r$ for a.e. $x\in\R^d$, such that the function $A$, defined in \eqref{eq:relation_between_a_and_A}, is a generalized $\Phi$-function. If $A$ is \hyperref[def:Dec]{$(\mathrm{Dec})_q$} with $q>1$, then
    \begin{equation*}
        A'(x, a(x, r)r)\leq (q-1)A(x, r).
    \end{equation*}
\end{lemma}
\begin{proof}
    Following the lines of \cite[Lemma 2.9]{bonder2019fractional}, to characterize $A'(x, a(x, r)r)$, we have to study the maximum value of $h(\ell)=a(x,r)r\ell-A(x,\ell)$ for all $\ell>0$, see \eqref{eq:convex_conjugate}. Computing the derivative of $h$, it is possible to prove from the convexity of $A(x,\ell)$ with respect to $\ell$, that $r$ is the point where $h$ attains its maximum value. Then
    \begin{equation*}
        A'(x, a(x, r)r)=a(x,r)r^2-A(x,r).
    \end{equation*}
    Since $A$ is \hyperref[def:Dec]{$(\mathrm{Dec})_q$}, by differentiating $A(x,\ell)\ell^{-q}$ with respect to $\ell$, it is possible to obtain that
    \begin{equation*}
        a(x,r)r^2\leq qA(x,r),
    \end{equation*}
    which concludes the proof.
\end{proof}

\begin{lemma}\label{lemma:dual_boundedness_elliptic}
    Let $a$ be as in the previous Lemma, and assume that $a(x,r)r$ is monotone with respect to $r$. If $\{\xi_n\}_{n\in\N}$ is sequence of functions uniformly bounded in $L^A(\R^d;\R^d)$, then there exists a constant $C$ independent of $n$ such that
    \begin{equation*}
        \|a(x,|\xi_n|)\xi_n\|_{L^{A'}(\R^d;\R^d)}\leq C.
    \end{equation*}
\end{lemma}
\begin{proof}
    Let $\{\xi_n\}_{n\in\N}$ be a uniformly bounded sequence in $L^A(\R^d;\R^d)$ and $\Psi\in L^A(\R^d;\R^d)$. From monotonicity of $a(x,r)r$ with respect to $r$, Young's inequality, convexity of $A$ with respect to the last variable and Lemma \ref{lemma:estimate_conjugate_derivative}, we obtain
    \begin{align*}
            &\int_{\R^d}{a(x,|\xi_n|)\xi_n\cdot \Psi}\,dx\leq \int_{\R^d}{a(x,|\xi_n|)\xi_n\cdot \xi_n}\,dx-\int_{\R^d}{a(x,|\xi_n|)\xi_n\cdot (\xi_n-\Psi)}\,dx\\
            &\qquad\leq \int_{\R^d}{\left(A'(x, a(x,|\xi_n|)|\xi_n|)+A(x, |\xi_n|)\right)}\,dx\\
            &\qquad\qquad\qquad\qquad+2\int_{\R^d}{\left(A'(x, a(x, |\xi_n|)|\xi_n|)+A(x, \frac{1}{2}|\xi_n-\Psi|)\right)}\,dx\\
            &\qquad\leq (3q-2)\int_{\R^d}{A(x,|\xi_n|)}\,dx+2\int_{\R^d}{A(x, \frac{1}{2}(\xi_n-\Psi))}\,dx\\
            &\qquad \leq (3q-1)\int_{\R^d}{A(x,|\xi_n|)}\,dx+\int_{\R^d}{A(x,\Psi)}\,dx.
    \end{align*}
    The result then follows from applying \eqref{eq:reverse_relation_norm_modular}.
\end{proof}

\begin{theorem}\label{theorem:continuous_dependence_s}
    Let $\Omega$ be a bounded open set and let $a$ and $A$ satisfy the hypothesis of Theorem \ref{thm:exitence_solutions}. Consider a sequence $\{s_n\}\subset (0,1]$ such that $s_n\to \sigma\in(0,1]$ and functions $\{\boldsymbol{f}_n\}\subset L^{A'}(\R^d;\R^d)$ such that $\boldsymbol{f}_n\to \boldsymbol{f}$ in $L^{A'}(\R^d;\R^d)$. Denoting by $\{u_n\}$ the sequence of solutions $u_n\in  \Lambda^{s_n,A}_0(\Omega)$ to
    \begin{equation}\label{eq:approximating_problems}
        \int_{\R^d}{a(x, D^{s_n}u_n)D^{s_n}u_n\cdot D^{s_n} w)}\,dx=\langle F_n, w\rangle_{\Lambda^{-s_n,A'}(\Omega)\times \Lambda^{s_n,A}_0(\Omega)}\quad\forall w\in \Lambda^{s_n,A}_0(\Omega),
    \end{equation}
    where $F_n=-D^{s_n} \boldsymbol{f}_n$ in the sense of the Proposition \ref{prop:characterization_dual}, one can extract a subsequence, still denoted $\{u_n\}$ such that
    \begin{equation*}
        u_n\to u \mbox{ in } L^A(\Omega) \quad\mbox{ and } \quad D^{s_n} u_n\rightharpoonup D^\sigma u \mbox{ in } L^A(\R^d;\R^d)
    \end{equation*}
    where $u\in\Lambda^{\sigma,A}_0(\Omega)$ solves
    \begin{equation*}
        \int_{\R^d}{a(x, D^{\sigma}u_n)D^{\sigma}u\cdot D^{s\sigma} w)}\,dx=\langle F, w\rangle_{\Lambda^{-\sigma,A'}(\Omega)\times \Lambda^{\sigma,A}_0(\Omega)}\quad\forall w\in \Lambda^{\sigma,A}_0(\Omega),
    \end{equation*}
    with $F=-D^\sigma \boldsymbol{f}$ in the sense of the Proposition \ref{prop:characterization_dual}.
\end{theorem}
\begin{proof}
    We start by observing that we have an uniform bound for $D^{s_n} u_n$ in $L^A(\R^d;\R^d)$ independent of $n$. Indeed, because $u_n$ is a solution of \eqref{eq:approximating_problems} and because of \eqref{eq:pointwise_relationship_between_a_and_A}, we have that for any given $\delta>0$ we have
    \begin{align*}
        &\|D^{s_n}u_n\|_{L^A(\R^d;\R^d)}\leq \int_{\R^d}{A(x, |D^{s_n}u_n|)}\,dx+1\\
        &\qquad \leq\frac{1}{1+a_-}\int_{\R^d}{a(x, |D^{s_n}u_n|)|D^{s_n}u_n|^2}\,dx+1\\
        &\qquad\leq \frac{1}{1+a_-}\langle F_n, u_n\rangle+1\\
        &\qquad\leq C_\delta\|\boldsymbol{f}_n\|_{L^{A'}(\R^d;\R^d)}+\delta\|D^{s_n}u_n\|_{L^A(\R^d;\R^d)}+1.
    \end{align*}
    Since $\sigma>0$, we can apply Theorem \ref{thm:convergence_bellido} to extract subsequences $\{s_n\}_{n\in\N}$ and $\{u_n\}_{n\in\N}$ such that
    \begin{equation*}
        u_n\to u \mbox{ in } L^A(\Omega) \quad\mbox{ and } \quad D^{s_n} u_n\rightharpoonup D^\sigma u \mbox{ in } L^A(\R^d;\R^d),
    \end{equation*}
    for some $u\in\Lambda^{\sigma,A}_0(\Omega)$.
    
    To prove that $u$ is a solution to the limit problem, let us fix a function $w\in\Lambda^{\sigma,A}_0(\Omega)$. For any $0<\varepsilon<1$, consider the function $u_\varepsilon=(1-\varepsilon)u+\varepsilon w\in \Lambda^{\sigma,A}_0(\Omega)$. Consider also, from Proposition \ref{prop:existence_approximating_function}, a sequence of functions $\{g_n\}$ with $g_n\in \Lambda^{s_n,A}_0(\Omega)$ such that $D^{s_n}g_n\to D^\sigma u$ in $L^A(\R^d;\R^s)$. From the monotonicity of $a(x,\xi)\xi$ with respect to $\xi$, it can be proved through algebraic manipulation that
    \begin{equation}\label{eq:liminf_inequality_with_epsilon_elliptic}
        \begin{aligned}
            &\varepsilon\liminf_{n\to\infty}{\int_{\R^d}{a(x, D^{s_n}u_n)D^{s_n}u_n\cdot D^\sigma (u-w)}\,dx}\\
            &\qquad\geq \varepsilon \int_{\R^d}{a(x, D^\sigma u_\varepsilon)D^\sigma u_\varepsilon\cdot D^\sigma (u-w)}\,dx\\
            &\qquad\qquad+\lim_{n\to\infty}{\int_{\R^d}{a(x, D^\sigma u_\varepsilon)D^\sigma u_\varepsilon\cdot (D^{s_n}u_n-D^\sigma u)}\,dx}\\
            &\qquad\qquad+\liminf_{n\to\infty}{\int_{\R^d}{a(x, D^{s_n}u_n)D^{s_n}u_n\cdot D^{s_n}(g_n-u_n)}\,dx}\\
            &\qquad\qquad-\lim_{n\to\infty}{\int_{\R^d}{a(x, D^{s_n}u_n)D^{s_n}u_n\cdot (D^{s_n}g_n-D^\sigma u)}\,dx}
        \end{aligned}
    \end{equation}
    Since $D^\sigma u_\varepsilon$ is uniformly bounded in $L^A(\R^d;\R^d)$ with respect to $0<\varepsilon<1$, Lemma \ref{lemma:dual_boundedness_elliptic} implies that $a(x,|D^\sigma u_\varepsilon|)D^\sigma  u_\varepsilon$ is uniformly bounded in $L^{A'}(\R^d;\R^d)$ with respect to $\varepsilon$. Therefore,
    \begin{equation*}
        \lim_{n\to\infty}{\int_0^t{\int_{\R^d}{a(x, D^\sigma u_\varepsilon)D^\sigma u_\varepsilon\cdot (D^{s_n}u_n-D^\sigma u)}\,dx}\,d\tau}=0
    \end{equation*}
    and
    \begin{equation*}
        \lim_{n\to\infty}{\int_0^t{\int_{\R^d}{a(x, D^{s_n}u_n)D^{s_n}u_n\cdot (D^{s_n}g_n-D^\sigma u)}\,dx}\,d\tau}=0.
    \end{equation*}
    because by construction $D^{s_n}u_n\rightharpoonup D^\sigma u$ in $L^A(\R^d;\R^d)$, and by definition of $g_n$ we have $D^{s_n}g_n\to D^\sigma u$ in $L^A(\R^d;\R^d)$.
    Moreover, since $u_n$ is a solution to \eqref{eq:approximating_problems} with the fractional parameter $s_n$, we have
    \begin{align*}
        &\lim_{n\to\infty}{\int_{\R^d}{a(x, D^{s_n}u_n)D^{s_n}u_n\cdot D^{s_n}(g_n-u_n)}\,dx}=\lim_{n\to\infty}\langle F_n,g_n-u_n\rangle=0.
    \end{align*}
    
    Applying these limits to \eqref{eq:liminf_inequality_with_epsilon_elliptic}, allows us to obtain
    \begin{multline*}
        \varepsilon\liminf_{n\to\infty}{\int_{\R^d}{a(x, D^{s_n}u_n)D^{s_n}u_n\cdot D^\sigma (u-w)}\,dx}\\\geq \varepsilon \int_{\R^d}{a(x, D^\sigma u_\varepsilon)D^\sigma u_\varepsilon\cdot D^\sigma (u-w)}\,dx.
    \end{multline*}
    Dividing both sides of the previous inequality by $\varepsilon$, and then taking the limit as $\varepsilon\to 0$, together with the fact that $-\Delta^\sigma_A:\Lambda^{\sigma,A}_0(\Omega)\to \Lambda^{-\sigma,A'}(\Omega)$ is hemicontinuous, we deduce that
    \begin{multline*}
        \liminf_{n\to\infty}{\int_{\R^d}{a(x, D^{s_n}u_n)D^{s_n}u_n\cdot D^\sigma(u-w)}\,dx}\\
        \geq\int_{\R^d}{a(x, D^\sigma u)D^\sigma u\cdot D^\sigma(u-w)}\,dx.
    \end{multline*}
    Since $D^{s_n} u_n\rightharpoonup D^\sigma u$ in $L^A(\R^d;\R^d)$, it is possible to obtain from algebraic manipulation that
    \begin{align*}
        &\liminf_{n\to\infty}{\int_{\R^d}{a(x, D^{s_n}u_n)D^{s_n}u_n\cdot (D^{s_n}u_n-D^\sigma w)}\,dx}\\
        &\qquad=\liminf_{n\to\infty}{\left(\int_{\R^d}{a(x, D^{s_n}u_n)D^{s_n}u_n\cdot (D^{s_n}u_n-D^\sigma u)}\,dx\right.}\\
        &\qquad\qquad\left.+\int_{\R^d}{a(x, D^{s_n}u_n)D^{s_n}u_n\cdot D^\sigma (u-w)}\,dx\right)\\
        &\qquad\geq \liminf_{n\to\infty}{\int_{\R^d}{(a(x, D^{s_n}u_n)D^{s_n}u_n-a(x, D^\sigma u)D^\sigma u)\cdot (D^{s_n}u_n-D^\sigma u)}\,dx}\\
        &\qquad\qquad+\lim_{n\to\infty}{\int_{\R^d}{a(x, D^\sigma u)D^\sigma u\cdot (D^{s_n}u_n-D^\sigma u)}\,dx}\\
        &\qquad\qquad\qquad+\liminf_{n\to\infty}{\int_{\R^d}{a(x, D^{s_n}u_n)D^{s_n}u_n\cdot D^\sigma (u-w)}\,dx}\\
        &\qquad\geq \lim_{n\to\infty}{\int_{\R^d}{a(x, D^\sigma u)D^\sigma u\cdot (D^{s_n}u_n-D^\sigma u)}\,dx}\\
        &\qquad\qquad+\liminf_{n\to\infty}{\int_{\R^d}{a(x, D^{s_n}u_n)D^{s_n}u_n\cdot D^\sigma (u-w)}\,dx}\\
        &\qquad\geq \lim_{n\to\infty}{\int_{\R^d}{a(x, D^\sigma u)D^\sigma u\cdot (D^{s_n}u_n-D^\sigma u)}\,dx}\\
        &\qquad\qquad+\int_{\R^d}{a(x, D^\sigma u)D^\sigma u\cdot D^\sigma (u-w)}\,dx\\
        &\qquad=\int_{\R^d}{a(x, D^\sigma u)D^\sigma u\cdot D^\sigma (u-w)}\,dx.
    \end{align*}
    Consequently, if we apply once again Proposition \ref{prop:existence_approximating_function} to construct a sequence of functions $\{w_n\}_{n\in\N}$ with $w_n\in\Lambda^{s_n,A}_0(\R^d;\R^d)$ such that $D^{s_n}w_n\to D^\sigma w$ in $L^A(\R^d;\R^d)$, Lemma \ref{lemma:dual_boundedness_elliptic} for the uniform boundedness of $a(x,|D^{s_n} u_n|)D^{s_n}u_n$ in $L^{A'}(\R^d;\R^d)$, and then use the fact that $u_n$ is a solution to \eqref{eq:approximating_problems}, we obtain that
    \begin{align*}
        &\int_{\R^d}{a(x, D^\sigma u)D^\sigma u\cdot D^\sigma(u-w)}\,dx\\
        &\qquad\leq \liminf_{n\to\infty}{\int_{\R^d}{a(x,D^{s_n} u_n)D^{s_n}u_n\cdot (D^{s_n}u_n-D^\sigma w)}\,dx}\\
        &\qquad=\lim_{n\to\infty}{\int_{\R^d}{a(x,D^{s_n} u_n)D^{s_n}u_n\cdot (D^{s_n}w_n-D^\sigma w)}\,dx}+\lim_{n\to\infty}{\langle F_n, u_n-w_n\rangle}\\
        &\qquad=\langle F, u-w\rangle.
    \end{align*}
    The proof then concludes from the fact that $w\in\Lambda^{\sigma,A}_0(\Omega)$ was chosen arbitrarily.
\end{proof}
\begin{remark}
    The existence and continuous dependence results obtained in  \cite{campos2023unilateral} for variational inequalities in $\Lambda^{s,p}_0(\Omega)$ can be generalized to the framework of the fractional generalized Orlicz-Sobolev spaces that we introduced in this paper.
\end{remark}

\textbf{Acknowledgments.} The authors' research was done under the PhD FCT-grant UI/BD/152276/2021.

\end{document}